\documentclass[11pt, a4paper, reqno]{amsart}
\usepackage{amsthm, amsmath, amsfonts, amssymb, appendix, dsfont, latexsym,mathrsfs,mathtools}
\usepackage{esint}
\usepackage{graphicx}
\usepackage{tikz}

\makeatletter
\usepackage{delarray, a4, color}
\usepackage{euscript}
\usepackage[latin1]{inputenc}
\usepackage{enumitem}
\usepackage{hyperref}
\hypersetup{
colorlinks,
menucolor=black,
linkcolor=black,
citecolor=black,
urlcolor=blue
}
\usepackage{cleveref}
\usepackage{crossreftools}
\definecolor{darkgreen}{rgb}{0,0.5,0}
\definecolor{darkred}{cmyk}{0,1,1,0.4}

\theoremstyle{plain}
\newtheorem{theorem}{Theorem}[section]
\crefname{maintheorem}{Theorem}{Theorems} 
\newtheorem*{theorem*}{Theorem}
\newtheorem{lemma}{Lemma}[section]
\newtheorem{corollary}{Corollary}[section]
\newtheorem{proposition}{Proposition}[section]

\theoremstyle{definition}

\newtheorem{remark}{Remark}
\newtheorem{remark*}{Remark\textup{*}}

\numberwithin{equation}{section}

\DeclareMathAlphabet{\mathpzc}{OT1}{pzc}{m}{it}

\def\C {\mathbb{C}}

\def\N {\mathbb{N}}

\def\R {\mathbb{R}}

\newcommand{\bA}{\mathbf{A}}
\newcommand{\bF}{\mathbf{F}}
\newcommand{\bJ}{\mathbf{J}}

\newcommand{\bx}{\mathbf{x}}
\newcommand{\by}{\mathbf{y}}

\newcommand{\cE}{\mathcal{E}}

\newcommand{\dd}{{\rm d}} 
\newcommand{\ii}{{\rm i}} 
\newcommand{\CH}{C_{\rm H}}

\DeclareMathOperator{\im}{\mathrm{Im}}

\DeclareMathOperator{\supp}{\mathrm{supp}}

\newcommand{\dist}{\mathrm{dist}}
\renewcommand{\div}{\operatorname{div}}

\newcommand{\loc}{\mathrm{loc}}

\usepackage{placeins}

\title[Global wellposedness of Chern-Simons-Schr\"odinger equation]{A Sharp condition on global wellposedness of Chern-Simons-Schr\"odinger equation}

\author[A. Ataei]{Alireza ATAEI}
\address{Alireza Ataei, Department of Mathematics, Uppsala University, Box 480, SE-751 06, Uppsala, Sweden}
\email{\url{alireza.ataei@math.uu.se}}


\allowdisplaybreaks
\begin{document}

\setcounter{tocdepth}{2}

\begin{abstract}
    In this work, we derive a sharp condition on the mass of the initial data for the global existence of the Chern-Simons-Schr\"odinger equation. As a corollary, we prove that if the strength of interaction is less than the Bogomolny bound, then, for a large enough mass of initial data, there exists a globally defined solution. On the other hand, for the interactions which are above the Bogomolny bound, the critical mass condition on the initial data for the global existence depends on the strength of the self-interacting field. Then, we show that the states with the initial critical mass and zero energy are standing wave solutions and globally well-posed. Moreover, they are static if the self-interacting field is large enough and non-static for small self-interacting field.
\end{abstract}
\maketitle
\section{Introduction}
The Chern-Simons-Schr\"odinger equation has its origin in quantum mechanics as a model for anyons and fractional quantum Hall effect; see \cite{ZhaHanKiv89,JaPi90,JaWe90,Wi90,EzHoIw91,EzHoIw91b, JaPi91, JaPi91b,Jackiw1991,Du95} and references therein. It has also been studied in mathematics as a challenging model which contains a non-local and non-homogenous self-generating magnetic energy; see \cite{BeBoSa95,Hu9,ByHuSe12,Hu13,LiSmTa14,OhPu15,LiSm16,ByHuSe16,Lim18,GoZh21,LiLiu22,KiKw231} and references therein. The model can be described as follows: Let $\beta \in \R_+, \gamma \in \R$, and $\phi \in H^1(\R^2)=H^1(\R^2;\C)$. For every $f \in L^1(\R^2)$ and $\bF \in L^1(\R^2;\C^2)$, define, in the principal value sense, the self-generating gauges
\begin{align*}
    \bA[f] &:= (\nabla^{\perp} \log |\cdot|) \ast f =\int_{\R^2} \frac{(\bx-\by)^{\perp}}{|\bx-\by|^2} f(\by) \, d \by, \\
     \bA \ast [\bF] &:=  \int_{\R^2} \frac{(\bx-\by)^{\perp}}{|\bx-\by|^2} \cdot  \bF(\by) \, d \by,
\end{align*}
where $(x_1,x_2)^{\perp} = (-x_2,x_1)$ for every $(x_1,x_2) \in \R^2$.  Moreover, we consider the definition of the currents
\begin{align*}
\bJ[\phi] &:= \im(\phi^{\ast} \nabla \phi),\\
    \textbf{j}_{\beta}[\phi] &:= \im (\phi^* (\nabla + \ii \beta \bA[|\phi|^2]) \phi)=\bJ[\phi] + \beta\bA[|\phi|^2] \, |\phi|^2.
\end{align*}
Now, we define the energy-functional 
\begin{align*}
  \cE_{\beta,\gamma}[\phi] := \int_{\R^2} |(\nabla + \ii \beta \bA[|\phi|^2]) \phi|^2 - \gamma \int_{\R^2} |\phi|^4.
\end{align*}
The parameter $\beta$ shows the strength of the self-interacting magnetic field and $\gamma$ shows the strength of the interaction between the particles, where $\gamma>0$ for attractive particles and $\gamma<0$ for repulsive particles. Now, for $T>0,$ we say that $\psi:[0,T] \times \R^2 \to \C$ satisfies the Chern-Simons-Schr\"odinger equation if $\psi \in L^{\infty}([0,T], H^1(\R^2)) \cap C([0,T],L^2(\R^2))$ and
\begin{equation}
\begin{aligned}
\label{eq:CSNLS}
\ii \partial_t \psi &= \frac{1}{2}  \frac{\partial}{\partial \psi^*} \cE_{\beta,\gamma}[\psi], \\&= - \frac{1}{2} (\nabla + \ii \beta \bA[|\psi|^2]) \cdot (\nabla + \ii \beta \bA[|\psi|^2]) \psi- \beta (\bA \ast  [\textbf{j}_{\beta}[\psi]]) \psi - \gamma |\psi|^2 \psi,
\end{aligned}
\end{equation}
weakly in $[0,T] \times \R^2$, where $\psi^* = \overline{\psi}$ is the complex conjugate of $\psi.$ We use the interpolation constant 
\begin{align*}
    \gamma_{\ast}(\beta) := \inf \left \{ \frac{\int_{\R^2} |(\nabla + \ii \beta \bA[|\phi|^2]) \phi|^2}{\int_{\R^2} |\phi|^4}: \phi \in H^1(\R^2), \int_{\R^2} |\phi|^2 = 1\right\},
\end{align*}
defined in \cite{AtLuThi24} to study the global existence of time-dependent solutions for \eqref{eq:CSNLS}.
\begin{theorem}
\label{thm:criticalconstant}
Let $\psi_0 \in H^1(\R^2) \setminus 0, \beta \in \R_+$, and $ \gamma \in \R$. If
\begin{align}
\label{eq:conditionglobalNLS}
\gamma  <  \frac{\gamma_{\ast}\left(\beta \|\psi_0\|_{L^2}^2 \right)}{\|\psi_0\|_{L^2}^2 },
\end{align}
then there exists a global solution $\psi \in L^{\infty}(\R_+, H^1(\R^2)) \cap C(\R_+,L^2(\R^2))$ of \eqref{eq:CSNLS}
in $\R_+ \times \R^2 $ with the initial data $\psi_0.$
Moreover, if 
\begin{align*}
\gamma >   \frac{\gamma_{\ast}(\alpha \beta)}{\alpha },
\end{align*}
for some $\alpha >0$, then there exist $\psi_0 \in C_c^{\infty}(\R^2)$, such that $\int_{\R^2} |\psi_0|^2 =\alpha$, and a local solution $\psi$ of \eqref{eq:CSNLS} in $[0,T] \times \R^2$ with the initial data $\psi_0$ for some $T>0$, which blows up at finite time, i.e., there exists a finite time $t^*>0$ such that 
\begin{align*}
\lim_{t \to t^*} \|\nabla \psi(t,\cdot)\|_{L^2} = \infty.
\end{align*}
\end{theorem}

\begin{remark}
    One of the strange properties of \eqref{eq:conditionglobalNLS} is the that if $\beta >0,$ $ \|\psi_0\|^2_{L^2} \geq \frac{2}{\beta},$ and $\gamma < 2 \pi \beta$, then there exists always a global solution for \eqref{eq:CSNLS} with the initial data $\psi_0$. Here, we are using $\gamma_{\ast}(\theta)= 2\pi \theta$ if $\theta \geq 2$; see \cite[Thm. 2]{AtLuThi24}. This is unlike the classical nonlinear Schr\"odinger equation, where the mass of the initial data should be small to have a global solution; see \cite{Weinstein-83}. 
\end{remark}
\begin{remark}
    If $\beta >0$ and $\gamma > 2 \pi \beta$, then for having a global solution for \eqref{eq:CSNLS} with the initial data $\psi_0$, it is required to have the mass condition $\|\psi_0\|^2_{L^2}< \frac{2}{\beta}$. Note that as $\beta \to \infty$ we get $\|\psi_0\|_{L ^2} \to 0.$ Hence, the smallness of $\|\psi_0\|_{L ^2}$ depends on $\beta$.
 \end{remark}

To mention some previous works on the global existence, the case of $\beta =0$ and $\gamma = 1$, the classical nonlinear Schr\"odinger equation, has been studied in the fundamental work of Weinstein \cite{Weinstein-83}. He proved in \cite[Thm. A,Thm. 4.1]{Weinstein-83} that a global solution to \eqref{eq:CSNLS} with the initial data $\psi_0 \in H^1(\R^2)$ exists if 
\begin{align*}
\|\psi_0\|_{L^2}^2 < \gamma_{\ast}(0), 
\end{align*}
and this bound is sharp, where $\gamma_{\ast}(0)$ is the Ladyzhenskaya--Gagliardo--Nirenberg interpolation constant. The global well-posedness for the case of general $\beta $ has been first studied in \cite{BeBoSa95}. The authors considered the lower-bound $\gamma_{\ast}(0)$ for $\gamma_{\ast}\left(\beta \|\psi_0\|_{L^2}^2 \right)$ in \eqref{eq:conditionglobalNLS}. Afterwards, there were several attempts to sharpen this result. The case of radially symmetric solutions and their global behavior was studied in \cite{ByHuSe12}. In \cite{OhPu15}, the global well-posedness has been proven with initial data which has small $H^2$ and a special weighted Sobolev norm.
In \cite{LiSm16,GoZh21}, the global well-posedness for the equivariant case has been studied. 

However, the literature is lacking the sharp condition for the global well-posedness of solutions to \eqref{eq:CSNLS}. To derive the sharp constant, we use that the interpolation constant $\gamma_{\ast}$ is continuous, the function $\frac{\gamma_{\ast}(\theta)}{\theta}$ is decreasing for $\theta>0,$ and a
 well-known argument in \cite{Weinstein-83} derives a uniform bound on the $H^1$-norm of an approximating solutions, constructed in  \cite{BeBoSa95}, for \eqref{eq:CSNLS}. Then, by a convergence argument, we obtain a global solution. We remark that there is a subtle difference between the way we pass to the limit in compare to \cite{BeBoSa95}, namely we do not use any estimates on the time derivative of the gauge term $\bA[|\psi|^2]$. Instead, we use a local convergence property for the gauge terms; see Lemma \ref{lem:localconvergauge}. Finally, the other part of the theorem follows as the solutions with negative energy blows up at finite time; see \cite[Thm 3.1]{BeBoSa95}.

In the second part, in Theorem \ref{thm:criticalconstant}, we consider the critical case that $\gamma = \frac{\gamma_{\ast}\left(\beta \|\psi_0\|_{L^2}^2 \right)}{\|\psi_0\|_{L^2}^2 }$ and $\frac{\psi_0}{\|\psi_0\|_{L^2}}$ is a minimizer for $\gamma_{\ast}\left( \beta \|\psi_0\|_{L^2}^2 \right)$. Then, $\cE_{\beta,\gamma}[\psi_0]=0$ and the following result is derived.

\begin{theorem}
\label{thm:criticalCSS}
    Let $\beta \in \R_+, \gamma \in \R $, and $\psi_0 \in H^1(\R^2) \setminus 0$ satisfy 
\begin{align*} 
   \cE_{\beta,\gamma}[\psi_0] = 0.
\end{align*}
Then, for every $T>0$, there exists a unique local solution $\psi \in  C([0,T],H^2(\R^2))$ of \eqref{eq:CSNLS} with the initial value $\psi_0$, which is of the form \begin{align}
\label{eq:criticalsolution}
    \psi(t,\cdot) = \psi_0(\cdot) \,e^{-\frac{\ii \, t \lambda_{\beta}(\psi_0)}{2} },
\end{align} 
for every $t \in [0,T]$, where
\begin{align*}
    \lambda_{\beta}(\psi_0) := \frac{1}{\|\psi_0\|^2_{L^2}} \left(\beta^2 \int_{\R^2} |\bA[|\psi_0|^2] \psi_0| ^2 - \int_{\R^2} |\nabla \psi_0|^2\right).
\end{align*}
    In particular, all the solutions of \eqref{eq:CSNLS} with the initial data $\psi_0$ are globally well-posed.
\end{theorem}

\begin{remark}
  Let $\beta \in \R_+, \gamma \in \R, T>0$, and $\psi_0 \in H^1(\R^2).$ Then, if  $\int_{\R^2} |\bx|^2 |\psi_0(\bx)|^2 \, \dd \bx < \infty $ and a solution $\psi \in C([0,T]; H^2(\R^2))$ of \eqref{eq:CSNLS} with the initial data $\psi_0$ is standing wave solution, i.e., 
  \begin{align*}
      \psi(t,\cdot) = e^{\ii \lambda t} \psi_0(\cdot),
  \end{align*}
  for $ 0\leq t \leq T$ and some $\lambda \in \R$, then, by \cite[Prop. 3.1]{BeBoSa95}, we have 
\begin{align*}
 \frac{1}{2} \cE_{\beta,\gamma}[\psi(t,\cdot)] = \frac{d ^2}{d t^2} \int_{\R^2} |\bx|^2 |\psi(t,\bx)|^2 \, \dd \bx = 0,
\end{align*}
for every $t >0.$ In particular, we obtain $\cE_{\beta,\gamma}[\psi_0]=0.$ This shows that standing wave solutions have zero energy. Moreover, by Theorem \ref{thm:criticalCSS}, all the solutions to \eqref{eq:CSNLS} with the initial data which has zero energy are standing wave solutions.
  
\end{remark}

\begin{remark}
    In the case of $\beta \geq 2,$ we prove that $\lambda_{\beta}(\psi_0)=0$ in Theorem \ref{thm:criticalCSS}; see Corollary \ref{cor:staticsolutions}. Hence, the solutions of \eqref{eq:CSNLS} are static, time-independent, for $\beta \geq 2.$ However, for $\beta < \sqrt{\frac{2}{3}}$, we obtain that $\lambda_{\beta}(\psi_0)<0$ in Theorem \ref{thm:criticalCSS}; see the proof of Corollary \ref{cor:staticsolutions}. Hence, for small enough $\beta$, there exists no static solution to \eqref{eq:CSNLS}.
\end{remark}

\begin{remark}
    In Theorem \ref{thm:criticalCSS}, it is crucial to start with an initial data $\psi_0$ which has zero energy. Otherwise, by assuming only $\gamma = \frac{\gamma_{\ast}\left(\beta \|\psi_0\|_{L^2}^2 \right)}{\|\psi_0\|_{L^2}^2 }$, by using conformal transformations, we may derive solutions which blow up in finite time; see \cite{Hu9,KiKwOh22,KiKw230,KiKw231}.  
\end{remark}

 \section{Acknowledgenment}
The author thanks Douglas Lundholm and Dinh-Thi Nguyen for many intriguing discussions on the work \cite{AtLuThi24} which has motivated the current work.

\section{Preliminary results}
In this section, we bring all the preliminary results for the proof of  \Cref{thm:criticalconstant,thm:criticalCSS}.
We use the notations $\N = \{1,2,3,\ldots\}$, $\R_+ = [0,\infty)$, $\Delta$ as the local Laplacian.

Now, we bring some basic inequalities on the magnetic energy.

\begin{lemma}\label{lem:basicinequalities}
Let $f \in H^1(\R^2)$. Then, the following inequalities hold.
\begin{enumerate}[label=(\roman*)]
\item {\bf Diamagnetic inequality.} For every $\beta \in \R$,
\begin{equation}\label{eq:AF-diamagnetic}
|\nabla |f|(\bx) | \leq  \left|(\nabla + {\rm i}\beta \bA\left[|f|^2\right]) f(\bx)\right|,
\end{equation}
for a.e. $\bx \in \R^2.$
\item {\bf Hardy-type inequality.} 
\begin{equation}\label{eq:AF-H1}
\int_{\R^2} \left|\bA\left[|f|^2\right]f\right|^2
\le \CH \left(\int_{\R^2} |f|^2\right)^2 \left(\int_{\R^2} \bigl|\nabla|f|\bigr|^2\right),
\end{equation}
for a universal constant $\CH>0$.
\item {\bf Gagliardo--Nirenberg interpolation inequality.} For every $p \geq 2$, 
\begin{align}\label{eq:gagliardonirenberg}
C(p) \int_{\R^2} |f|^p \leq \left(\int_{\R^2} |f|^2\right) \left(\int_{\R^2} |\nabla f|^2\right)^{\frac{p-2}{2}},
\end{align}
where $C(p)$ is a universal constant depending on $p.$
\end{enumerate}
\end{lemma}
\begin{proof}
     We refer to \cite[Thm.~7.21]{LieLos01} for \eqref{eq:AF-diamagnetic}, \cite[Lem.~3.4]{LunRou15} for \eqref{eq:AF-H1}, and \cite{Ga59,Ni59} for \eqref{eq:gagliardonirenberg}.
\end{proof}
 We remark that the constant in \eqref{eq:AF-H1} was given by $\CH=\frac{3}{2}$ in \cite[Lemma~3.4]{LunRou15}. However, it might not be optimal as mentioned in \cite[Remark~3.7(ii)]{HofLapTid-08}.
 
Next, we bring the following conservation of mass and energy. 

\begin{proposition}
\label{prop:conservationlaws}
Let $(\lambda,\beta,\gamma) \in \R_+ \times \R^2$, $T>0,$ and $\psi_0 \in H^2(\R^2)$. Assume that $\psi \in C([0,T],H^2(\R^2))$ is a weak solution of 
\begin{align*}
\ii \partial_t \psi + \ii\lambda \Delta^2 \partial_t \psi = \frac{1}{2}  \frac{\partial}{\partial \psi^*} \cE_{\beta,\gamma}[\psi], 
\end{align*}
in $[0,T] \times \R^2$ with initial data $\psi_0.$ Then, for every $0\leq t \leq T$, we have
\begin{align}
\label{eq:conservationofmass}
   \int_{\R^2} |\psi(t,\cdot)|^2 + \lambda \int_{\R^2} |\Delta \psi(t,\cdot)|^2 &= \int_{\R^2} |\psi_0|^2 + \lambda \int_{\R^2} |\Delta \psi_0|^2,
\end{align}
and 
\begin{align}
\label{eq:conservationofenergy}
      \cE_{\beta,\gamma}[\psi(t,\cdot)] &= \cE_{\beta,\gamma}[\psi_0].
\end{align}
\begin{proof}
    The proof of \eqref{eq:conservationofenergy} follows the same argument as \cite[Prop 2.1]{BeBoSa95}, so we avoid repetition. Finally, to demonstrate \eqref{eq:conservationofmass}, we can multiply the equation for $\psi$ with $\psi^*$ and integrate the imaginary part, as mentioned in \cite[Page 245]{BeBoSa95}.
\end{proof}

\end{proposition}

Now, we demonstrate some well-known properties of the operators $(1-\lambda \Delta)^{-1}, (1+\lambda \Delta^2) ^{-1}$ for $\lambda \in \R_+$.
\begin{lemma}
\label{lem:basicfouriertypeestimate}
    Let $f \in H^1(\R^2)$ and $\lambda  \in \R_+$. Then, 
    \begin{enumerate}[label=\text{(\roman*)}]
        \item 
       $$ \|(1-\lambda \Delta)^{-1} f\|_{L^2(\R^2)} \leq \|f\|_{L^2(\R^2)},$$
        \item $$\| \nabla (1-\lambda \Delta)^{-1} f\|_{L^2(\R^2)}
\leq \|\nabla f\|_{L^2(\R^2)},$$ 
\item   $$\|\Delta (1+\lambda \Delta^2)^{-1} f\|_{H^{-1}(\R^2)} \leq  \|\nabla f\|_{L^2(\R^2)}.$$
        \end{enumerate}

\end{lemma}
\begin{proof}
The first two inequalities are simple applications of the Fourier transform.
To prove the last inequality, we note that
\begin{align*}
    \|\Delta (1+\lambda \Delta^2)^{-1} f\|_{H^{-1}(\R^2)} &= \|\div \left (\nabla (1+\lambda \Delta^2)^{-1} f \right)\|_{H^{-1}(\R^2)} \\& \leq  \|\nabla (1+\lambda \Delta^2)^{-1} f \|_{L^2(\R^2)}\\& \leq \|\nabla f\|_{L^2(\R^2)}, 
\end{align*}
where the last inequality is a simple application of the Fourier transform.
    
\end{proof}

The next two propositions drive the $L^{\infty}$ and $L^p$ estimates on the operators $\bA , \bA \ast$.

\begin{proposition}
\label{prop:LinfinityestimateonA}
    Let $f \in L^1 \cap L^p(\R^2)$ for some $p>2.$ Then, $\bA[f] \in L^{\infty}(\R^2)$ and
    \begin{align*}
        \|\bA[f]\|_{L^{\infty}} \leq \|f\|_{L^1} +\left(\frac{2 \pi}{2-p'} \right)^{\frac{1}{p'}}\|f\|_{L^p},
    \end{align*}
where $p':= \frac{p}{p-1}$.

\end{proposition}
\begin{proof}
 For every $\bx \in \R^2,$ we have 
\begin{align*}
   |\bA[f](\bx)| &= \biggl | \int_{\R^2} \frac{\bx - \by}{ |\bx - \by|^2} f(\by) \, \dd \by  \biggr |
   \\ &\leq  \biggl | \int_{\R^2 \setminus B(\bx,1)} \frac{\bx - \by}{ |\bx - \by|^2} f(\by)\, \dd \by  \biggr | + \biggl | \int_{ B(\bx,1)} \frac{\bx - \by}{ |\bx - \by|^2} f(\by)\, \dd \by  \biggr | 
   \\ & \leq \|f\|_{L^1} + \biggl ( \int_{B(\bx,1)} \frac{1}{|\bx-\by|^{p'}} \, \dd \by \biggr)^{\frac{1}{p'}}  \|f\|_{L^p(B(\bx,1))} \\
   & \leq \|f\|_{L^1} +\left(\frac{2 \pi}{2-p'}\right)^{\frac{1}{p'}} \|f\|_{L^p}
 \end{align*}
where we used H\"older's inequality on the one to the last inequality.

\end{proof}

\begin{proposition}
\label{prop:Youngtypeineq}
    Let $f \in L^p(\R^2)$ and $\bF \in L^p(\R^2;\C^2)$ for some $1<p<2.$ Then,  both $\bA[f]$ and $ \bA \ast [\bF]$ belong to $ L^r(\R^2)$ for $r := \frac{2p}{2-p}$ and 
    \begin{align*}
        \left \|\bA[f] \right\|_{L^r}  \leq C(r) \|f\|_{L^p},\\
        \left \| \bA \ast [\bF] \right\|_{L^r} \leq C(r) \left\|\bF \right\|_{L^p}, 
    \end{align*}
where $C(r)$ is a universal constant depending on $r.$

\end{proposition}
\begin{proof}
    Both inequalities are a simple application of the weak-type Young's inequality (see, e.g., \cite[Section~4.3]{LieLos01}), by noting that 
\begin{align*}
     & \|\bA [f]\|_{L^r} \leq \left \| \frac{1}{|\bx|} \ast |f| \right\|_{L^r} \leq C  \left\|\frac{1}{|\bx|} \right\|_{L^{2,w}} \|f\|_{L^p} \\
    & \|\bA \ast [\bF]\|_{L^r} \leq \left \| \frac{1}{|\bx|} \ast |\bF| \right\|_{L^r} \leq C  \left\|\frac{1}{|\bx|} \right\|_{L^{2,w}} \|\bF\|_{L^p},\\
\end{align*}
    where $C$ is a universal constant depending on $r.$
\end{proof}
Finally, we need the following result on the local convergence of gauges.
\begin{lemma}
\label{lem:localconvergauge}
   Let $p>2, T>0$, and $f_n \in L^1 \cap L^p_{\loc}([0,T] \times \R^2)$, $\bF_n \in L^1 \cap L^p_{\loc}([0,T] \times \R^2;\C^2)$ be sequences which converge to $f, \bF$ in  $L^p_{\loc}([0,T] \times \R^2)$, $L^p_{loc}([0,T] \times \R^2;\C^2)$, respectively. Moreover, assume that  
   $\sup_{0 \leq t \leq T}\|f_n(t,\cdot) \|_{L^1(\R^2)}+ \|\bF_n(t,\cdot)\|_{L^1(\R^2)}$ is uniformly bounded. Then, for every $0<q<p $,  $\bA[f_n]$, $\bA \ast [\bF_n]$ converge to $\bA[f]$, $\bA \ast [\bF]$ in $L^q_{\loc}([0,T] \times \R^2;\C^2)$, $L^q_{\loc}([0,T] \times \R^2)$, respectively.
\end{lemma}
\begin{proof}
    Let $B \subset \R^2$ be a finite disk centered at the origin with radius $R_B>0$, $g_n := f_n-f$, and $r,q>0$. Define $B_r := (1+r)B$ which is the dilation of the disk $B$ with the scale $1+r.$ Then, 
\begin{align*}
 & \int_0^T  \int_{B} |\bA[g_n]|^q = \int_0^T \int_B \left |\int_{\R^2} \frac{\bx - \by}{|\bx-\by|^2} \,g_n(t,\by) \, \dd \by \right|^q \, \dd \bx  \dd t
    \\ & \leq 2^q    \int_0^T  \int_B \left |\int_{\R^2 \setminus B_r} \frac{\bx - \by}{|\bx-\by|^2} \,g_n(t,\by) \, \dd \by \right|^q \, \dd \bx \dd t + 2^q   \int_0^T  \int_B \left |\int_{ B_r} \frac{\bx - \by}{|\bx-\by|^2} \,g_n(t,\by) \, \dd \by \right|^q \, \dd \bx \dd t
    \\ & \leq \frac{2^q |B|}{\left(r R_B\right)^q} \int_{0}^T \|g_n(t,\cdot)\|^q_{L^1} \, \dd t +2^q \int_{0}^T \int_{ B} \left |\int_{B_r} \frac{\bx - \by}{|\bx-\by|^2} \,g_n(t,\by) \, \dd \by \right|^q \, \dd \bx \dd t
    \\ &  \leq \frac{2^q |B|}{\left(r R_B\right)^q} \int_{0}^T \|g_n(t,\cdot)\|^q_{L^1} \, \dd t+C(r,p,q,B) \int_0^T \|g_n(t,\cdot)\|^q_{L^p(B_r)} \, \dd t
    \\ & \leq \frac{2^q |B| T}{\left(r R_B\right)^q} \sup_{0\leq t \leq T} \|g_n(t,\cdot)\|^q_{L^1} +C(r,p,q,B) T^{\frac{p-q}{p}} \|g_n\|^q_{L^p([0,T] \times B_r)}
 \end{align*}
where we used  H\"older's inequality in the last two inequalities and, by $p>2$,
\begin{align*}
    C(r,p,q,B) := 2^q \int_{B} \left(\int_{B_r} \frac{1}{|\bx - \by|^{\frac{p}{p-1}}} \, \dd \by \right)^{\frac{q(p-1)}{p}} \, \dd \bx < \infty.
\end{align*}
    Now, first letting $n \to \infty$, we arrive at 
    \begin{align*}
        \limsup_{n \to \infty}  \int_{B} |\bA[g_n]|^q \leq  \frac{2^q |B| T}{\left(r R_B\right)^q} \sup_{0\leq t \leq T} \|g_n(t,\cdot)\|^q_{L^1}.
    \end{align*}
Hence, since $\sup_{0 \leq t \leq T}\|g_n(t,\cdot)\|_{L^1}$ is uniformly bounded and $r>0$ is arbitrary, we derive that 
$\bA[f_n]$ converges to $\bA[f]$ in  $L^q_{\loc}([0,T] \times \R^2;\C^2).$ Likewise, we can prove that
$\bA \ast [F_n]$ converges to $\bA \ast [F]$ in  $L^q_{\loc}([0,T] \times \R^2).$

\end{proof}

\section{Proof of the theorem}

\begin{proof}[Proof of Theorem \ref{thm:criticalconstant}]
   We follow the argument in the proof of \cite[Thm 2.3]{BeBoSa95}.  Assume that
  \eqref{eq:conditionglobalNLS} holds for $\psi_0 \in H^1(\R^2) \setminus 0$. Let $\varepsilon>0$ and $$\psi^{\varepsilon}_0 := (1 - \varepsilon^{\frac{1}{4}} \Delta)^{-1} \psi_0.$$ Then, by Lemma \ref{lem:basicfouriertypeestimate},
  \begin{align}
  \label{eq:uniformL2initialdata}
      \|\psi_0^{\varepsilon}\| ^2_{L^2(\R^2)} \leq  \|\psi_0\| ^2_{L^2(\R^2)}, \quad \|\nabla \psi_0^{\varepsilon}\| ^2_{L^2(\R^2)} \leq \|\nabla \psi_0\|^2_{L^2(\R^2)},
  \end{align}
  Moreover,  by the same argument as \cite[Lem. 2.2]{BeBoSa95}, there exists a unique $\psi^{\varepsilon} \in C(\R_+,H^2(\R^2))$ solution of 
  \begin{equation}
  \begin{aligned}
  \label{eq:approximatingsolution}
      &\ii \partial_t \psi^{\varepsilon} + \ii \varepsilon \Delta^2 \partial_t \psi^{\varepsilon} = \frac{1}{2} \frac{\partial}{\partial (\psi^{\varepsilon})^*} \cE_{\beta,\gamma}[\psi^{\varepsilon}] \\
      & =- \frac{1}{2} (\nabla + \ii \bA[|\psi^{\varepsilon}|^2]) \cdot (\nabla + \ii \bA[|\psi^{\varepsilon}|^2]) \psi^{\varepsilon}- \beta (\bA \ast \textbf{j}_{\beta}[[\psi^{\varepsilon}]]) \psi^{\varepsilon} - \gamma |\psi^{\varepsilon}|^2 \psi^{\varepsilon}
      ,
  \end{aligned}
  \end{equation}
in $\R_+ \times \R^2$, with the initial value $\psi_0^{\varepsilon}.$
      Then, by the assumption \eqref{eq:conditionglobalNLS} and the continuity of $\gamma_{\ast}$; see \cite[Prop. 3.26]{AtLuThi24}, for $\delta >0, \varepsilon_0>0$ small enough, we have
      \begin{equation}
\begin{aligned}
\label{eq:choiceofepsilon}
\gamma (1+\delta) <  \frac{\gamma_{\ast}\left(\beta (1+ \sqrt{\varepsilon})  \|\psi_0\|_{L^2}^2 \right)}{(1+ \sqrt{\varepsilon})  \|\psi_0\|_{L^2}^2 },
\end{aligned}
\end{equation}
for every $0<\varepsilon<\varepsilon_0.$
 In the rest of the argument, we consider $0<\varepsilon<\max\left(\varepsilon_0,1\right)$ such that \eqref{eq:choiceofepsilon} holds.

Now, we derive a global bound on $\|\psi^{\varepsilon}(t,\cdot)\|_{H^1(\R^2)}$.
    To prove the bound, we first use the conservation of the mass and energy in Proposition \ref{prop:conservationlaws} to imply that
    \begin{equation}
    \label{eq:conservationlaws}
\begin{aligned}
\int_{\R^2} |\psi^{\varepsilon}(t,\cdot)|^2 + \varepsilon \int_{\R^2} |\Delta \psi^{\varepsilon}(t,\cdot)|^2 &= \int_{\R^2} |\psi^{\varepsilon}_0|^2 + \varepsilon \int_{\R^2} |\Delta \psi^{\varepsilon}_0|^2,\\
    \cE_{\beta,\gamma}[\psi^{\varepsilon}(t,\cdot)] &= \cE_{\beta,\gamma}[\psi^{\varepsilon}_0],
\end{aligned}
\end{equation}
for every $ t \in \R_+$. Then, by applying the Fourier transform again, we obtain 
\begin{equation}
\begin{aligned}
\label{eq:uniformL2bound}
 \|\psi^{\varepsilon}(t,\cdot)\|_{L^2}^2 &\leq \|\psi_0\|^2_{L^2} + \sqrt{\varepsilon}   \left\| \varepsilon^{\frac{1}{4}} \Delta \psi_0^{\varepsilon}\right\|^2_{L^2}  \\&\leq \|\psi_0\|^2_{L^2}+ \sqrt{\varepsilon}    \left\|\psi_0\right\|^2_{L^2} = (1+ \sqrt{\varepsilon}) \|\psi_0\|^2_{L^2},
\end{aligned}
\end{equation}
for every $ t \in \R_+.$ Hence, by using that  $\frac{\gamma_{\ast}(\theta)}{\theta}$ is a decreasing function of $\theta>0$; see \cite[Lem. 3.18]{AtLuThi24}, we derive that 
\begin{align*}
\frac{\gamma_{\ast}\left(\beta  \|\psi^{\varepsilon}(t,\cdot)\|^2_{L^2} \right)}{\|\psi^{\varepsilon}(t,\cdot)\|^2_{L^2}}
    \geq \frac{\gamma_{\ast}\left(\beta (1+ \sqrt{\varepsilon})    \|\psi_0\|^2_{L^2} \right)}{(1+ \sqrt{\varepsilon})  
 \|\psi_0\|^2_{L^2}}.
\end{align*}
for every $t \in \R_+.$ Therefore, by \eqref{eq:choiceofepsilon}, it is implied that
\begin{align}
\label{eq:boundongammastar}
    \frac{\gamma_{\ast}\left(\beta  \|\psi^{\varepsilon}(t,\cdot)\|^2_{L^2} \right)}{\|\psi^{\varepsilon}(t,\cdot)\|^2_{L^2}} > \gamma (1+\delta)
\end{align}
for every $t \in \R_+.$ Then, by \eqref{eq:conservationlaws} and \eqref{eq:boundongammastar}, we obtain
\begin{align*}
   \cE_{\beta}[\psi^{\varepsilon}(t,\cdot)]
    &=  \gamma \int_{\R^2} |\psi^{\varepsilon}(t,\cdot)|^4 + \cE_{\beta,\gamma}[\psi^{\varepsilon}_0] \\ &\leq   \frac{\gamma \, \|\psi^{\varepsilon}(t, \cdot)\|^2_{L^2}}{\gamma_{\ast}\left(\beta \|\psi^{\varepsilon}(t,\cdot)\|^2_{L^2}  \right)} \, \cE_{\beta}[\psi^{\varepsilon}(t,\cdot)] +  \cE_{\beta,\gamma}[\psi^{\varepsilon}_0]
    \\ &=  \frac{1}{1+ \delta} \,\cE_{\beta}[\psi^{\varepsilon}(t,\cdot)] +  \cE_{\beta,\gamma}[\psi^{\varepsilon}_0].
\end{align*}
In conclusion, 
\begin{align*}
  \cE_{\beta}[\psi^{\varepsilon}(t,\cdot)] &\leq \frac{\delta+1}{\delta} \left |\cE_{\beta,\gamma}[\psi^{\varepsilon}_0]\right |,\\ \int_{\R^2} |\nabla \psi^{\varepsilon}(t,\cdot)|^2 & \leq  \left( 1 + \sqrt{\frac{3}{2}}\beta \int_{\R^2}|\psi^{\varepsilon}_0|^2 \right)^2  \cE_{\beta}[\psi^{\varepsilon}(t,\cdot)],
\end{align*}
where we used \cite[Lem. 3.3]{AtLuThi24} in the last inequality.
Therefore, by \eqref{eq:uniformL2initialdata}, \eqref{eq:uniformL2bound}, and the Gagliardo-Nirenberg interpolation inequality \eqref{eq:gagliardonirenberg},  $\| \psi^{\varepsilon}(t,\cdot)\|_{H^1}$ is uniformly bounded and
\begin{equation}
\begin{aligned}
    \label{eq:boundonH1}
    \|\psi^{\varepsilon}(t,\cdot)\|^2_{H^1}& \leq 2 \|\psi_0\|^2_{L^2} + 
 \frac{\delta+1}{\delta}\,
 \left( 1 + \sqrt{\frac{3}{2}}\,\beta \int_{\R^2}|\psi_0|^2 \right)^2 \left |\cE_{\beta,\gamma}[\psi^{\varepsilon}_0] \right |  \leq  C_1,
\end{aligned}
\end{equation}
for all $t \in \R_+ $, where $C_1$ is a constant depending on $\beta,\gamma,\delta,\|\psi_0\|_{H^1}$. Now, we aim to prove that every term on the right-hand-side of \eqref{eq:approximatingsolution} is uniformly bounded in $$L^2_{\loc}(\R_+,H^{-1}(\R^2)) + L^2_{\loc}(\R_+ \times \R^2).$$
To do so, we need to show that $\left\|\bA[|\psi^{\varepsilon}(t,\cdot)|^2]\right\|_{L^2}, \left\|\bA[|\psi^{\varepsilon}(t,\cdot)|^2]\right\|_{L^4},  \left\| \bA \ast [\textbf{j}_{\beta}[\psi^{\varepsilon}(t,\cdot)]]\right\|_{L^4} $ are uniformly bounded for $t \in \R_+.$

To estimate $\left\|\bA[|\psi^{\varepsilon}(t,\cdot)|^2]\right\|_{L^2}$, we use Diamagnetic inequality \eqref{eq:AF-diamagnetic}, Hardy-type inequality \eqref{eq:AF-H1}, and \eqref{eq:boundonH1}, to obtain that 
\begin{align*}
  \left \|\bA[|\psi^{\varepsilon}(t,\cdot)|^2]\right\|^2_{L^2}
    \leq \frac{3}{2} \left(\int_{\R^2} |\psi^{\varepsilon}(t,\cdot)|^2\right)^2 \left(\int_{\R^2} \bigl|\nabla\psi^{\varepsilon}(t,\cdot) \bigr|^2\right) \leq \frac{
3}{2} C_1^3,
\end{align*}
for every $t \in \R_+$. 

For $\left\|\bA[|\psi^{\varepsilon}(t,\cdot)|^2]\right\|_{L^4}$, by Proposition \ref{prop:Youngtypeineq} and the Gagliardo-Nirenberg interpolation inequality \eqref{eq:gagliardonirenberg}, we have 
\begin{equation}
\label{eq:Lpgaugeestimate}
\begin{aligned}
\left\|\bA[|\psi^{\varepsilon}(t,\cdot)|^2]\right\|_{L^4} &\leq C_2   \|\psi^{\varepsilon}(t,\cdot)\|^2_{L^{\frac{8}{3}}(\R^2)}  \\&\leq C_2 \, C_3 \|\psi^{\varepsilon}(t,\cdot)\|^{\frac{1}{2}}_{L^2} \|\nabla \psi^{\varepsilon}(t,\cdot)\|^{\frac{1}{3}}_{L^2} \leq 
C_1^{\frac{5}{12}}\, C_2 \, C_3  .
\end{aligned}
\end{equation}
for every $t \in \R_+$, where $C_2,C_3$ are universal constants. Moreover, 
\begin{align*}
 &  \left\| \bA \ast [\textbf{j}_{\beta}[\psi^{\varepsilon}(t,\cdot)]]\right\|_{L^4} \leq  C_2  \|\textbf{j}_{\beta}[\psi^{\varepsilon}(t,\cdot)]\|_{L^{\frac{4}{3}}} \\&\leq   C_2  \left(\|\bJ[\psi^{\varepsilon}(t,\cdot)]\|_{L^{\frac{4}{3}}} + \beta \left\|\bA[|\psi^{\varepsilon}(t,\cdot)| ^2] |\psi^{\varepsilon}(t,\cdot)|^2 \right\|_{L^{\frac{4}{3}}}\right)\\ 
   &\leq  C_2  \left(\|\psi^{\varepsilon}(t,\cdot)\|_{L^{4}} \|\nabla \psi^{\varepsilon}(t,\cdot)\|_{L^2} +\beta \left\|\bA[|\psi^{\varepsilon}(t,\cdot)| ^2]  \right\|_{L^{4}} \left\| \psi^{\varepsilon}(t,\cdot)  \right\|^2_{L^{4}}\right)
   \\ & \leq  C_2  \left(C_4\|\psi^{\varepsilon}(t,\cdot)\|^{\frac{1}{2}}_{L^{2}} \|\nabla \psi^{\varepsilon}(t,\cdot)\|^{\frac{3}{2}}_{L^2} +\beta C_1^{\frac{1}{3}}\, C_2 \, C_3  C_4\|\psi^{\varepsilon}(t,\cdot)\|_{L^{2}} \|\nabla \psi^{\varepsilon}(t,\cdot)\|_{L^2}\right)
   \\ &\leq    C_1\, C_2\, C_4 +\beta C_1^{\frac{4}{3}}\, C_2^2 \, C_3  C_4.\, 
\end{align*}
Here, we used Proposition \ref{prop:Youngtypeineq} on the first inequality, H\"older's inequality on the third inequality, \eqref{eq:Lpgaugeestimate} and the Gagliardo-Nirenberg interpolation inequality \eqref{eq:gagliardonirenberg} on the fourth inequality, and \eqref{eq:boundonH1} on the last inequality. In conclusion, by Lemma \ref{lem:basicfouriertypeestimate} $(iii)$ and \eqref{eq:approximatingsolution}, we derive that $\partial_t \psi ^{\varepsilon}$ is uniformly bounded in  $L^2_{\loc}(\R_+,H^{-1}(\R^2)),$ and $\psi^{\varepsilon}$ is uniformly bounded in $L^{\infty}(\R_+; H^1(\R^2))$. Then, by a classical theorem of Lions, see \cite{Lions}, we conclude that, up to a subsequence,  $\psi^{\varepsilon}$ converges to $\psi  \in  L^{\infty}(\R_+, H^1(\R^2)) \cap C(\R_+,L^2(\R^2))$ weakly in $L^2_{\loc}(\R_+,H^1(\R^2))$, pointwise a.e. in $\R_+ \times \R^2$, and strongly in the space $L^q_{\loc}(\R_+ \times \R^2)$ for every $1 \leq q< \infty$.

Now, we claim that $\psi$ satisfies \eqref{eq:CSNLS} weakly in $\R_+ \times \R^2$ with the initial data $\psi_0.$ To prove the claim, by using the Fourier transform, we have the convergence of $\psi^{\varepsilon}_0$ to $\psi_0$ in $H^1(\R^2)$, which derives that $\psi(0,\cdot) = \psi_0(\cdot).$ Hence, it is sufficient to prove that the right hand side of \eqref{eq:approximatingsolution} converges to the right hand side of \eqref{eq:CSNLS} in the distribution sense. Let $T>0$ and $\varphi \in C([0,T];C^{\infty}_c(\R^2))$ be a test function. 
Since, by \eqref{eq:uniformL2bound}, $\sup_{t \in \R_+} \left\| | \psi_{\varepsilon}(t,\cdot)|^2 \right\|_{L^1}$ is uniformly bounded and $\psi_{\varepsilon}$ converges to $\psi$ in $L^q_{\loc}(\R_+ \times \R^2)$ for every $1 \leq q < \infty$, by Lemma \ref{lem:localconvergauge}, we derive that $\bA[|\psi_{\varepsilon}|^2]$ converges to $\bA[|\psi|^2]$ in $L^p_{\loc}(\R_+ \times \R^2)$ for every $0<p< \infty$. Then, 
\begin{align*}
   &\lim_{\varepsilon \to 0}  \int_0^T \int_{\R^2}
- (\nabla + \ii \beta \bA[|\psi^{\varepsilon}|^2]) \cdot (\nabla + \ii \beta \bA[|\psi^{\varepsilon}|^2]) \psi^{\varepsilon}
   \, \overline{\varphi} \\&= \lim_{\varepsilon \to 0} \int_0^T \int_{\R^2} (\nabla + \ii \beta \bA[|\psi^{\varepsilon}|^2]) \psi^{\varepsilon} \cdot \overline{ (\nabla + \ii \beta \bA[|\psi^{\varepsilon}|^2]) \varphi}
   \\&=\int_0^T \int_{\R^2} (\nabla + \ii \beta \bA[|\psi|^2]) \psi \cdot \overline{ (\nabla + \ii \beta \bA[|\psi|^2]) \varphi},
\end{align*}
where on the last equality we use the weak convergence of $\psi^{\epsilon}$ to $\psi$ in $L^2_{\loc}(\R_+,H^1(\R^2))$.

To show the converges of the term $(\bA \ast \textbf{j}_{\beta}[[\psi^{\varepsilon}]]) \psi^{\varepsilon}$ in the distribution, we need more work. First, we note that $\textbf{j}_{\beta}[\psi^{\varepsilon}] = \bJ[\psi^{\varepsilon}] + \beta \bA[|\psi^{\varepsilon}| ^2] |\psi^{\varepsilon}|^2.$ Hence, it is enough to prove the convergence of $(\bA \ast [\bJ[\psi^{\varepsilon}]]) \psi^{\varepsilon} $ and $(\bA \ast [\bA[|\psi^{\varepsilon}| ^2] |\psi^{\varepsilon}|^2]) \psi^{\varepsilon} $ separately. For the second one, we use Lemma \ref{lem:localconvergauge}. To prove the conditions of  Lemma \ref{lem:localconvergauge}, we consider
\begin{align*}
   \sup_{t \in \R_+} \int_{\R^2} \left|\bA[|\psi^{\varepsilon}(t,\cdot)| ^2]\right| |\psi^{\varepsilon}(t,\cdot)|^2 &\leq   \sup_{t \in \R_+} \left \|\bA[|\psi^{\varepsilon}(t,\cdot)| ^2]\right \|_{L^4} \left\|\psi^{\varepsilon}(t,\cdot)\right\|^2_{L^{\frac{8}{3}}}
    \\ & \leq  C_1^{\frac{1}{3}}\, C_2 \, C_3 \,  \sup_{t \in \R_+} \|\psi^{\varepsilon}(t,\cdot)\|^{\frac{3}{2}}_{L^2} \, \|\nabla \psi^{\varepsilon}(t,\cdot)\|^{\frac{1}{2}}_{L^2} 
    \\ & \leq  C_1^{\frac{4}{3}}\, C_2 \, C_3,
\end{align*}
where H\"older's inequality is used on the first inequality, \eqref{eq:Lpgaugeestimate} and the Gagliardo-Nirenberg's inequality \eqref{eq:gagliardonirenberg} are used on the second inequality, and \eqref{eq:boundonH1} is used on the last inequality. Moreover, since $\bA[|\psi^{\varepsilon}| ^2],\psi^{\varepsilon}$ converges to $\bA[|\psi|^2],\psi$ in $L^p_{\loc}(\R_+ \times \R^2;\R^2),L^p_{\loc}(\R_+ \times \R^2)$, respectively, for every $1\leq p < \infty$,  we derive that $\bA[|\psi^{\varepsilon}| ^2] |\psi^{\varepsilon}|^2$ converges to $\bA[|\psi| ^2] |\psi|^2$ in  $L^q_{\loc}(\R_+ \times \R^2;\R^2)$ for every $1 \leq q < \infty$. Hence, by Lemma \ref{lem:localconvergauge}, it is implied that  $\bA \ast [\bA[|\psi^{\varepsilon}| ^2] |\psi^{\varepsilon}|^2]$ converges to $\bA \ast [\bA[|\psi| ^2] |\psi|^2]$ in $L^p_{\loc}(\R_+ \times \R^2)$ for every $0<p < \infty$.

Now, it is remained to prove the convergence of $(\bA \ast [\bJ[\psi^{\varepsilon}]]) \psi^{\varepsilon}$ in the sense of distribution.
To see this, we first use Lemma \ref{lem:localconvergauge} again to derive that $\bA\left[\psi^{\varepsilon}  \, \varphi \right]$ converges strongly to $\bA\left[\psi  \, \varphi \right]$ in $L^p_{\loc}([0,T] \times \R^2)$ for every $0<p<\infty$. Then, by the weak convergence of $\psi^{\epsilon}$ to $\psi$ in $L^2_{\loc}(\R_+,H^1(\R^2))$, we have
\begin{align*} \lim_{\varepsilon \to 0}   \int_0^T  \int_{\R^2}
(\bA \ast [\bJ[\psi^{\varepsilon}]]) \psi^{\varepsilon} 
   \, \varphi &= -\lim_{\varepsilon \to 0} \int_0^T  \int_{\R^2}
\bJ[\psi^{\varepsilon}] \cdot \bA\left[\psi^{\varepsilon} 
   \, \varphi \right] \\
  &= -\lim_{\varepsilon \to 0} \int_0^T  \int_{B}
\bJ[\psi^{\varepsilon}] \cdot \bA\left[\psi^{\varepsilon} 
   \, \varphi \right] \\& -\lim_{\varepsilon \to 0} \int_0^T  \int_{\R^2 \setminus B}
\bJ[\psi^{\varepsilon}] \cdot \bA\left[\psi^{\varepsilon} 
   \, \varphi \right]
   \\ & =-\int_0^T  \int_{B}
\bJ[\psi] \cdot \bA\left[\psi
   \, \varphi \right] \\& -\lim_{\varepsilon \to 0} \int_0^T  \int_{\R^2 \setminus B}
\bJ[\psi^{\varepsilon}] \cdot \bA\left[\psi^{\varepsilon} 
   \, \varphi \right],
\end{align*}
for every disk $B \subset \R^2.$ Now, we prove that the second term decreases uniformly as $|B| \to \infty$. Assume that the disk $B$ is large enough such that 
\begin{align*}
    \dist\left(\supp(\varphi),\R^2 \setminus B \right) \geq \frac{1}{\delta}
\end{align*}
for $\delta >0.$ Then, by the Cauchy-Schwarz inequality and \eqref{eq:boundonH1},
\begin{align*}
 & \left |   \int_0^T  \int_{\R^2 \setminus B}
\bJ[\psi^{\varepsilon}] \cdot \bA\left[\psi^{\varepsilon} 
   \, \varphi \right] \right | \\& \leq   \int_0^T  \int_{\R^2 \setminus B} |\nabla \psi^{\varepsilon}(t,\bx)| \, |\psi^{\varepsilon}(t,\bx)|\, \left(\int_{\R^2} \frac{1}{|\bx -\by|} |\psi^{\varepsilon}(t,\by)|\, |\varphi(t,\by)| \dd \by\right) \, \dd \bx \dd t
   \\ & \leq \delta \int_0^T \int_{\R^2 \setminus B}|\nabla \psi^{\varepsilon}(t,\cdot)| \, |\psi^{\varepsilon}(t,\cdot)| \, \dd t \int_{\R^2} |\psi^{\varepsilon}(t,\cdot)|\, |\varphi(t,\cdot)| \, \dd t
   \\ & \leq  \delta \int_0^T \|\nabla \psi^{\varepsilon}(t,\cdot)\|_{L^2} \| \psi^{\varepsilon}(t,\cdot)\|^2_{L^2} \| \varphi(t,\cdot)\|_{L^2} \, \dd t \leq \delta  C_1^{\frac{3}{2}} \,  \int_0^T \| \varphi(t,\cdot)\|_{L^2} \, \dd t.  
\end{align*}
Moreover, the same approximation holds by replacing $\psi^{\varepsilon}$ with $\psi$.
Hence, letting $\delta \to 0$, we obtain that
\begin{align*}
     \lim_{\varepsilon \to 0}   \int_0^T  \int_{\R^2}
(\bA \ast [\bJ[\psi^{\varepsilon}]]) \psi^{\varepsilon} 
   \, \varphi =     \int_0^T  \int_{\R^2}
(\bA \ast [\bJ[\psi]]) \psi
   \, \varphi.
\end{align*}
Finally, the term $|\psi^{\varepsilon}|^2 \psi^{\varepsilon} $ in \eqref{eq:approximatingsolution} converges to $|\psi|^2 \psi $  in $L^p_{\loc}(\R_+ \times \R^2)$ for every $p \geq 1$, which completes the proof of the claim.

Now, to prove the other side of the theorem, assume that $\alpha >0$ such that 
\begin{align*}
    \gamma > \frac{\gamma_{\ast}(\alpha \beta)}{\alpha}.
\end{align*}
    Then, by definition of $\gamma_{\ast}(\alpha \beta)$, there exists a function $\phi \in C^{\infty}_c(\R^2)$ which satisfies $\int_{\R^2} |\phi|^2=1$ and
\begin{align*}
    \cE_{\alpha \beta, \alpha \gamma}[\phi] < 0.
\end{align*}
Now, let $\psi_0 := \sqrt{\alpha} \, \phi $ and $\psi \in C([0,T],H^2(\R^2))$ be a local solution of \eqref{eq:CSNLS} in $\R_+ \times \R^2$ with the initial data $\psi_0$; see \cite[Thm 2.1]{BeBoSa95} for the existence of such solutions. We claim that $\psi$ blows up at some finite time. By \cite[Thm 3.1]{BeBoSa95}, it is enough to prove that 
\begin{align*}
    \cE_{\beta,\gamma}[\psi_0] <0. 
\end{align*}
To prove this, we use the scaling property of $\cE_{\beta,\gamma}$ and
\begin{align*}
     \cE_{\beta,\gamma}[\psi_0] =  \cE_{\beta,\gamma}[\sqrt{\alpha} \, \phi] =\alpha \, \cE_{\alpha \, \beta, \alpha \, \gamma}[\phi] < 0.
\end{align*}

\end{proof}

\begin{corollary}
    Let $\psi_0 \in H^1(\R^2), \beta>0, \gamma \in \R_+$, such that $\|\psi_0\|_{L^2}^2 \geq \frac{2}{\beta}$. Then, there exists a global solution $\psi$ to \eqref{eq:CSNLS} in $\R_+ \times \R^2$ with the initial data $\psi_0$ if $\gamma < 2\pi \beta.$ 
\end{corollary}
\begin{proof}
    By \cite[Thm. 2]{AtLuThi24}, we have 
    \begin{align*}
        \frac{\gamma_{\ast}\left(\beta \|\psi_0\|_{L^2}^2 \right)}{\|\psi_0\|_{L^2}^2 } =2\pi \beta.
    \end{align*}
  Hence, the proof follows from  \Cref{thm:criticalconstant}.
\end{proof}

\section{Critical case}

\begin{lemma}
  \label{lem:H2boundminimizer}
    Let $\psi \in H^1(\R^2)$ be a minimizer for $\gamma_{\ast}(\beta)$. Then, $\psi \in H^{2}(\R^2)$.
\end{lemma}
\begin{proof}
    First, by \cite[Lem. 3.10]{AtLuThi24}, we have $\psi \in C^{\infty}(\R^2)$ and satisfies 
\begin{align*}
      \Big[-\left(\nabla + {\rm i}\beta\bA\left[|\psi|^2\right]\right)\cdot \left(\nabla + {\rm i}\beta\bA\left[|\psi|^2\right]\right) - 2 \beta \bA \ast [\textbf{j}_{\beta}[\psi]] - 2\gamma_{*}(\beta)|\psi|^2\Big] \psi = \lambda \psi,
\end{align*}
    in $\R^2$, where 
\begin{align*}
    \lambda = \beta^2 \int |\bA[|\psi|^2] \psi|^2 - \int_{\R^2} |\nabla \psi|^2. 
\end{align*}
    Hence,
\begin{align}
\label{eq:shorterformofequa}
\Delta u = \sum_{i=1}^6 I_i,
\end{align}
where
\begin{align*}
I_1 &:= \beta^2\left|\bA\left[|\psi|^2\right]\right|^2 \psi, \\
I_2 &:= - 2\beta^2 \bA 
*\left[|\psi|^2 \bA\left[|\psi|^2\right]\right] \psi,\\
I_3 &:= - 2\beta \bA * \left[\bJ[\psi]\right] \psi,\\
I_4 &:=  - 2\ii \beta \bA\left[|\psi|^2\right] \cdot \nabla \psi,\\
I_5 &:=  - 2\gamma_*(\beta)|\psi|^2\psi,\\
I_6 &:= -  \lambda \psi.
\end{align*}
Since $\psi \in H^1(\R^2)$, to prove the lemma it is enough to demonstrate $\Delta \psi \in L^2(\R^2).$ We prove that each term in \eqref{eq:shorterformofequa} belongs to $L^2(\R^2)$. Obviously, $I_6 \in L^2(\R^2)$ and, by the Gagliardo-Nirenberg interpolation inequality \eqref{eq:gagliardonirenberg}, $I_5 \in L^2(\R^2).$ Now, since, by the Gagliardo-Nirenberg interpolation inequality \eqref{eq:gagliardonirenberg},  $\psi \in L^p(\R^2)$ for every $p \geq 2$, by Proposition \ref{prop:LinfinityestimateonA}, we derive that $$\bA[|\psi|^2] \in L^{\infty}(\R^2),$$ and $I_1, I_4 \in L^2(\R^2).$ For the term $I_2$, we have
\begin{align*}
  \frac{1}{2\beta^2}  \|I_2\|_{L^2}
  &\leq \left\| \bA \ast [|\psi|^2 \bA[|\psi|^2]]  \right\|_{L^4} \,\|\psi\|_{L^4} 
  \\&\leq C_1 \left\||\psi|^2 \bA[|\psi|^2] \right\|_{L^{\frac{4}{3}}} \|\psi\|_{L^4}
  \\ & \leq C_1  \left\|\bA[|\psi|^2]\right\|_{L^{\infty}} \|\psi  \|^2_{L^{\frac{8}{3}}} \|\psi\|_{L^4}
  \\ & \leq C_2 \left\|\bA[|\psi|^2]\right\|_{L^{\infty}} \|\psi\|^{\frac{3}{2}}_{L^2} \|\nabla \psi\|_{L^2},
\end{align*}for some universal constants $C_1,C_2>0$. Here, we use H\"older's inequality on the first inequality, Proposition \ref{prop:Youngtypeineq} on the second inequality, and the Gagliardo-Nirenberg interpolation inequality \eqref{eq:gagliardonirenberg} on the last inequality. Similarly, we consider
\begin{align*}
   \frac{1}{2\beta} \|I_3\|_{L^2} &\leq  \left\| \bA \ast [\bJ[\psi]]  \right\|_{L^4} \,\|\psi\|_{L^4}
   \\ &\leq  C_1  \|\bJ[\psi]\|_{L^{\frac{4}{3}}} \|\psi\|_{L^4}
   \\ &\leq C_1  \|\nabla \psi\|_{L^2}  \|\psi\|^2_{L^4}
   \\ &\leq C_3 \|\nabla \psi\|^2_{L^2} \|\psi\|_{L^2},
\end{align*}
for a universal constant $C_3>0,$ where H\"older's inequality is applied on the first and the third inequalities, Proposition \ref{prop:Youngtypeineq} is applied on the second inequality, and the Gagliardo-Nirenberg interpolation inequality \eqref{eq:gagliardonirenberg} is applied on the last inequality. This completes the proof.

\end{proof}
Now, we prove Theorem \eqref{thm:criticalCSS}.

\begin{proof}[Proof of Theorem \ref{thm:criticalCSS}]
    First, by the assumptions, $\phi_0 :=\frac{\psi_0}{\|\psi_0\|_{L^2}}$ is a minimizer for $\gamma_{\ast}(\alpha) = \gamma \|\psi_0\|^2_{L^2}$ where $$\alpha:=\beta \|\psi_0\|^2_{L^2}.$$ Hence, by \cite[Lem. 3.10]{AtLuThi24} and \Cref{lem:H2boundminimizer}, we derive that $\phi_0 \in C^{\infty}(\R^2) \cap H^2(\R^2)$ and satisfy
\begin{equation}\label{eq:EL}
\Big[-\left(\nabla + {\rm i}\alpha\bA\left[|\phi_0|^2\right]\right)\cdot \left(\nabla + {\rm i}\alpha\bA\left[|\phi_0|^2\right]\right) - 2 \alpha \bA \ast \textbf{j}_{\alpha}[\phi_0] - 2\gamma_{*}(\alpha)|\phi_0|^2\Big] \phi_0 = \lambda_{\beta}(\psi_0) \phi_0,
\end{equation}
in $\R^2.$ In conclusion, by \cite[Thm. 2.1]{BeBoSa95}, there exists a unique local solution $\psi \in C([0,T],H^2(\R^2))$ of \eqref{eq:CSNLS} with the initial data $\psi_0$. Moreover, by Proposition \ref{prop:conservationlaws}, we derive that 
\begin{align*}
\|\psi(t,\cdot)\|_{L^2} &= \|\psi_0\|_{L^2},\\
    \cE_{\beta, \gamma}[\psi(t,\cdot)] &=  \cE_{\beta, \gamma}[\psi_0] = 0,
\end{align*}
for every $0 \leq t \leq T.$
Hence, for every $0 \leq t \leq T$ , $\phi(t,\cdot) := \frac{\psi(t,\cdot)}{\|\psi_0\|_{L^2}}$ is also a minimizer for $\gamma_{\ast}(\alpha)$
and, by using \cite[Lem. 3.10]{AtLuThi24} again, satisfies the variational equation 
\begin{align*}
    \Big[-\left(\nabla + {\rm i}\alpha\bA\left[|\phi(t,\cdot)|^2\right]\right)\cdot \left(\nabla + {\rm i}\alpha\bA\left[|\phi(t,\cdot)|^2\right]\right) - 2 \alpha \bA \ast \textbf{j}_{\alpha}[\phi(t,\cdot)] \\- 2\gamma_{*}(\alpha)|\phi(t,\cdot)|^2\Big] \phi(t,\cdot) = \lambda(t) \phi(t,\cdot),
\end{align*}
in $\R^2$,
where 
\begin{align*}
    \lambda(t) &:=\alpha^2 \int_{\R^2} |\bA[|\phi(t,\cdot)| ^2] \phi(t,\cdot)|^2 - \int_{\R^2} |\nabla \phi(t,\cdot)|^2 \\& = \frac{1}{\|\psi_0\|^2_{L^2}} \left(\beta^2 \int_{\R^2} |\bA[|\psi(t,\cdot)|^2] \psi(t,\cdot)| ^2 - \int_{\R^2} |\nabla \psi(t,\cdot)|^2\right).
 \end{align*}
 Then, by combining the variational equation for $\phi(t,\cdot)$ with \eqref{eq:CSNLS}, we obtain that 
\begin{align*}
   \partial_t \psi(t,\cdot) = \frac{- \ii \lambda(t)}{2} \psi(t,\cdot) 
\end{align*}
for every $0<t < T$. By solving the equation, we conclude that 
\begin{align*}
    \psi(t,\cdot) = \psi_0 \, e^{-\frac{\ii}{2} \int_{0}^t \lambda(s) \, \dd s},
\end{align*}
for every $0 \leq t \leq T$. Hence, for every $0 \leq t \leq T$,
\begin{align*}
    \lambda(t) = \frac{1}{\|\psi_0\|^2_{L^2}} \left(\beta^2 \int_{\R^2} |\bA[|\psi_0|^2] \psi_0| ^2 - \int_{\R^2} |\nabla \psi_0|^2\right) = \lambda_{\beta}(\psi_0),
\end{align*}
which concludes the proof of \eqref{eq:criticalsolution}.
    
\end{proof}

The next lemma gives a more concrete way to compute $\lambda_{\beta}(\psi_0)$ in Theorem \ref{thm:criticalCSS}.

\begin{lemma}
\label{lem:lambdacomputation}
 Let $\beta \in \R_+$ and $\psi_0 \in H^1(\R^2)$ be a minimizer for $\gamma_{\ast}(\beta)$. Assume that $\gamma_{\ast}$ has a derivative at $\beta.$  Then,
\begin{align*}
    \beta^2 \int_{\R^2} |\bA[|\psi_0|^2] \psi_0| ^2 - \int_{\R^2} |\nabla \psi_0|^2 = \left(\beta \gamma'_{\ast}(\beta)- \gamma_{\ast}(\beta)\right) \int_{\R^2} |\psi_0|^4.
\end{align*}

\end{lemma}

\begin{proof}
    Consider the function 
\begin{align*}
    f(\theta) &:=  \cE_{\theta,\gamma_{\ast}(\theta)}[\psi_0] = \int_{\R^2} |(\nabla + \ii \theta \bA[|\psi_0|^2]) \psi_0|^2 - \gamma_{\ast}(\theta) \int_{\R^2} |\psi_0|^4
    \\ = & \int_{\R^2} |\psi_0|^2 + \theta^2 \int_{\R^2} |\bA[|\psi_0|^2] \psi_0|^2 + 2 \theta \int_{\R^2} \bA[|\psi_0|^2] \cdot \bJ[\psi_0] - \gamma_{\ast}(\theta) \int_{\R^2} |\psi_0|^4
\end{align*}
for $\theta \in \R.$ Then, $f$ is non-negative, takes its minimum value $0$ at $\theta = \beta$, and is differentiable at $\theta = \beta$. Hence, 
\begin{align*}
    \frac{d}{d \theta} \bigg|_{\theta = \beta}  f(\theta) = 0.
\end{align*}
In conclusion,
\begin{align*}
    2 \beta \int_{\R^2} |\bA[|\psi_0|^2] \psi_0|^2 + 2 \int_{\R^2} \bA[|\psi_0|^2] \cdot \bJ[\psi_0] - \gamma'_{\ast}(\beta) \int_{\R^2} |\psi_0|^4 = 0.
\end{align*}
By combining this with $f(\beta)= 0$, we derive that 
\begin{align*}
    0 &= \int_{\R^2} |\psi_0|^2 + \beta^2 \int_{\R^2} |\bA[|\psi_0|^2] \psi_0|^2 + 2 \beta \int_{\R^2} \bA[|\psi_0|^2] \cdot \bJ[\psi_0] - \gamma_{\ast}(\beta) \int_{\R^2} |\psi_0|^4 \\
    &= \int_{\R^2} |\psi_0|^2 - \beta^2 \int_{\R^2} |\bA[|\psi_0|^2] \psi_0|^2 +  (\beta \gamma'_{\ast}(\beta)-\gamma_{\ast}(\beta)) \int_{\R^2} |\psi_0|^4, 
\end{align*}
which concludes the proof.

\end{proof}
\begin{proposition}
\label{prop:derivativeat2}
    For every $\beta \geq 2$, we have $\gamma_{\ast}$ is differentiable at $\beta$ and
\begin{align*}
    \gamma'_{\ast}(\beta) = 2\pi.
\end{align*}
    
\end{proposition}
\begin{proof}
    If $\beta >2$, then the proof follows by $\gamma_{\ast}(\beta) = 2\pi \beta$ for every $\beta \geq 2$; see \cite[Thm. 2]{AtLuThi24}. For the case of $\beta =2$, we use \cite[Prop. 3.34]{AtLuThi24} to derive the following estimate 
\begin{align*}
  2 \pi \theta \leq  \gamma_{\ast} (\theta) \leq 2 \pi + \frac{\pi}{2} \theta^2,
\end{align*}
 for every $\theta \in \R_+ .$ Hence,
 \begin{align*}
       \frac{\pi}{2}(\theta +2) \leq \frac{\gamma_{\ast}(\theta) - \gamma_{\ast}(2)}{\theta -2} \leq 2 \pi
 \end{align*}
for every $\theta <2$ and 
\begin{align*}
 2 \pi \leq   \frac{\gamma_{\ast}(\theta) - \gamma_{\ast}(2)}{\theta -2} \leq \frac{\pi}{2}(\theta +2),
\end{align*}
    for every $\theta >2.$ In conclusion,
\begin{align*}
  \gamma'_{\ast}(2) =  \lim_{\theta \to 2} \frac{\gamma_{\ast}(\theta) - \gamma_{\ast}(2)}{\theta -2} = 2 \pi.
\end{align*}

\end{proof}

\begin{corollary}
\label{cor:staticsolutions}
  Let $\beta, \gamma  \in \R$, and $\psi_0 \in H^1(\R^2)$ such that $\cE_{\beta,\gamma}[\psi_0]=0$. If $\beta \|\psi_0\|^2_{L^2} \geq 2$, then $\beta \|\psi_0\|^2_{L^2} \in 2\N$ and all the solutions $\psi \in C([0,T],H^2(\R^2))$ to \eqref{eq:CSNLS} with the initial data $\psi_0$ are static, i.e., $\psi = \psi_0.$ Moreover, if $\beta \|\psi_0\| ^2_{L^2} < \sqrt{\frac{1}{C_H}}$, then there exists no static solution  $\psi \in C([0,T],H^2(\R^2))$ to \eqref{eq:CSNLS} with the initial data $\psi_0$ for any $T>0$.
\end{corollary}

\begin{proof}
 First, we consider the case that $\beta \|\psi_0\|^2_{L^2} \geq 2$. Then, $\frac{\psi_0}{\|\psi_0\|_{L^2}}$ is a minimizer for $\gamma_{\ast}\left(\beta \|\psi_0\|^2_{L^2} \right)$. Hence, by \cite[Thm. 2]{AtLuThi24}, we imply that $\beta \|\psi_0\|^2_{L^2} \in 2 \N$. Now, assume that $\psi \in C([0,T],H^2(\R^2))$ is a solution to \eqref{eq:CSNLS} with the initial data $\psi_0.$ Then, by Theorem \ref{thm:criticalCSS}, we have 
    \begin{align*}
        \psi(t,\cdot) = \psi_0 \, e^{-\frac{\ii \lambda_{\beta}(\psi_0) t}{2}}, 
    \end{align*}
    for every $t \in [0,T]$, where we remind that
\begin{align*}
    \lambda_{\beta}(\psi_0) =  \frac{1}{\|\psi_0\|^2_{L^2}} \left(\beta^2 \int_{\R^2} |\bA[|\psi_0|^2] \psi_0| ^2 - \int_{\R^2} |\nabla \psi_0|^2\right).
\end{align*}
    Moreover, by Proposition \ref{prop:derivativeat2} and Lemma \ref{lem:lambdacomputation}, we have $\lambda_{\beta}(\psi_0)=0$. Hence, $
          \psi(t,\cdot) = \psi_0.$
This completes the first part. For the second part, assume that $\beta \|\psi_0\|^2_{L^2}< \sqrt{\frac{1}{\CH}}$. By Diamagnetic inequality \eqref{eq:AF-diamagnetic} and Hardy-type inequality \eqref{eq:AF-H1}, we have
\begin{align*}
    \int_{\R^2} |\nabla \psi_0|^2 \geq \frac{1}{\CH \|\psi_0\|^4_{L^2}} \int_{\R^2} |\bA[|\psi_0|^2] \psi_0| ^2.
\end{align*}
Hence, by Theorem \ref{thm:criticalCSS}, we derive that all the solutions $\psi$ of \eqref{eq:CSNLS} are of the form
  \begin{align*}
        \psi(t,\cdot) = \psi_0 \, e^{-\frac{\ii \lambda_{\beta}(\psi_0) t}{2}}, 
    \end{align*}
    for every $t \in [0,T]$, where 
\begin{align*}
    \lambda_{\beta}(\psi_0) = \frac{1}{\|\psi_0\|^2_{L^2}} \left( \beta^2 \int_{\R^2} |\bA[|\psi_0|^2] \psi_0| ^2 - \int_{\R^2} |\nabla \psi_0|^2 \right) <0.
\end{align*}
This completes the second part.
    
\end{proof}

\newcommand{\etalchar}[1]{$^{#1}$}

\end{document}